\documentclass[a4paper]{article}
\usepackage[latin1]{inputenc}
\usepackage{full page}
\usepackage{amsmath,amsthm,url}

\title{Octonion associators}
\author{Stephen J. Sangwine\thanks{Stephen J. Sangwine is with the
        School of Computer Science and Electronic Engineering,
        University of Essex, Wivenhoe Park, Colchester, CO4 3SQ, UK.
        Email: \protect\url{sjs@essex.ac.uk}.}}
\newtheorem{theorem}{Theorem}
\newtheorem{definition}{Definition}
\newcommand*{\conjugate}[1]{\ensuremath{\overline{#1}}}
\begin{document}
\maketitle
\begin{abstract}
The algebra of octonions is non-associative (as well as non-commutative).
This makes it very difficult to derive algebraic results, and to perform
computation with octonions.
Given a product of more than two octonions, in general, the order of
evaluation of the product (placement of parentheses) affects the result.
Inspired by the concept of the commutator $[x,y]= x^{-1}y^{-1}xy$
we show that an associator can be defined that
multiplies the result from one evaluation order to give the result from
a different evaluation order. For example, for the case of three arbitrary
octonions $x$, $y$ and $z$ we have $((xy)z)a = x(yz)$,
where $a$ is the associator \emph{in this case}.
For completeness, we include other definitions of the commutator,
$[x,y]=xy-yx$ and associator $[x,y,z]=(xy)z-x(yz)$,
which are well known,
although not particularly useful as algebraic tools.
We conclude the paper by showing how to extend the concept of the multiplicative
associator to products of four or more octonions, where the number of evaluation
orders is greater than two.
\end{abstract}
\section{Introduction}
\label{sec:intro}
In a non-commutative algebra, such as the quaternions or octonions \cite{Ward:1997,10.1090/S0273-0979-01-00934-X,ConwaySmith},
where the result of the product of two arbitrary elements $x$ and $y$ differs (in general) according to order,
it is natural to ask: how does the product $xy$ differ from $yx$?
Equally, in a non-associative algebra, specifically, in this paper, the octonions,
how does $x(yz)$ differ from $(xy)z$?
A cursory search will reveal the existence of two quantities obtained
by taking the numerical difference in each case.
These are called the \emph{commutator} and the \emph{associator}
and are defined as follows:
\begin{definition}[Commutator \cite{CollinsDictMaths}]
$[x,y] = xy - yx$
\end{definition}
\begin{definition}[Associator \cite{10.1090/S0273-0979-01-00934-X}]
$[x,y,z] = x(yz) - (xy)z$
\end{definition}
These quantities, in a sense, measure whether their arguments commute (or associate).
If they do, the commutator (or associator) yields a zero result.
In the case of quaternions, of course, the associator is identically zero because
quaternion multiplication is associative.
Any four octonions satisfy the following \emph{associator identity} \cite[\S\,2.4, p.\,13]{Schafer}:
\[
a[x,y,z] + [a,x,y]z = [ax,y,z] - [a,xy,z] + [a,x,yz].
\]

Our interest in this paper, however, is not the numerical difference between the two results
(of multiplication in different orders),
but rather the question of \emph{how to convert one result into the other},
particularly for the case of products of three or more octonions, which we address in
the next section.
Before doing this however, we review the following result for the conversion of the
product of two non-commuting elements of an algebra into the result of the product
taken in the opposite order.
Confusingly, this is also known as a commutator\,\footnote{An alternative definition is possible,
by interchanging $x$ and $y$ on the right.}:
\begin{definition}[Commutator \cite{CollinsDictMaths}]
\label{def:mulcommutator}
$[x,y] = y^{-1} x^{-1} y x$
\end{definition}
If we write $[x,y] = c$ according to this definition, then we have $xyc = yx$,
since $xyc = x y\,y^{-1} x^{-1}\,y x = x x^{-1}\,y x = y x$.
That is, the commutator $c$ is a value that multiplies the product in one order
to yield the result of the product taken in the other order.
A clearer way to understand how this commutator works is to note that $y^{-1}x^{-1}$
is the inverse of $xy$.
Hence we see that the commutator as defined in Definition \ref{def:mulcommutator}
operates by multiplying $xy$ by the product of its inverse with $yx$.
It therefore converts one product into the other by converting the first product
into unity, then multiplying unity by the second product.
This is trivial, once seen.

The fact that this definition works for octonions as well as quaternions
(which appears not to have been noted until now)
depends on bi-associativity
\cite[\S\,11.9]{Salzman}, which says that in any product involving only $x$, $y$, their
conjugates or inverses, and real scalars\footnote{Although we are not considering complexified
octonions in this paper, `real scalars' could be replaced with `complex scalars' without loss
of validity.}, the order of evaluation is immaterial
(that is parentheses can be dropped).
This means that, for arbitrary non-zero octonions $x$ and $y$ we have:
$(xy)(y^{-1}x^{-1}) = xyy^{-1}x^{-1} = 1$, hence the inverse of $xy$ is $y^{-1}x^{-1}$,
a result that we will use in the sequel in Theorem \ref{thm:associator}.

The commutator that works in the opposite direction is the quaternion/octonion conjugate of the one
defined above.
This is easily seen from $xyc = yx$ by multiplying on the right by the conjugate of $c$
(denoted here by an overbar)
giving $xyc\conjugate{c} = yx\conjugate{c}$.
Since the commutator has unit norm, the product of $c$ with its conjugate is 1
and we find $xy = yx\conjugate{c}$.

\section{A `multiplicative' octonion associator}
In this section we show that it is possible to construct an octonion associator in
a manner similar to the commutator of Definition \ref{def:mulcommutator}.
\begin{theorem}
\label{thm:associator}
Given three arbitrary octonions $x$, $y$ and $z$,
they may be multiplied together in the order $x, y, z$ in two ways: $(xy)z$ or $x(yz)$.
These two products are related by the associator $[x,y,z]$ defined as follows:
\begin{equation}
\label{eq:associator}
[x,y,z] = (z^{-1}(y^{-1} x^{-1}))(x(yz)) 
\end{equation}
such that
\begin{align}
\label{eq:assoc1}
((xy)z)[x,y,z] &= x(yz)\\
\label{eq:assoc2}
(xy)z          &= (x(yz))\conjugate{[x,y,z]}.
\end{align}
\end{theorem}
\begin{proof}
We first show that
\begin{align}
\label{eq:toshow}
((xy)z)^{-1} &= (z^{-1}(y^{-1} x^{-1}))\\
\intertext{As we noted in the previous section, $(xy)^{-1} = y^{-1} x^{-1}$. Writing $w=xy$ therefore,
\eqref{eq:toshow} can be written as:}
(w z)^{-1} &= z^{-1}w^{-1}
\end{align}
and the result follows.

Now we can rewrite \eqref{eq:assoc1} as follows, by substitution of the right-hand side of
\eqref{eq:associator}, then replacement of the right-hand side of \eqref{eq:toshow} by $((xy)z)^{-1}$:
\begin{align}
((xy)z)[((xy)z)^{-1}x(yz)] &= x(yz)
\end{align}
Here we have an expression involving only two terms: $(xy)z$ and $x(yz)$, and their inverses.
Hence bi-associativity holds and the \emph{square} brackets may be dropped.
The left-hand side then reduces to the right-hand side.

The proof of \eqref{eq:assoc2} follows by the same method.
\end{proof}
Note that this associator can be understood in the same way as the commutator of Definition \ref{def:mulcommutator},
namely as the product of the inverse of one evaluation order times the result of the other evaluation order.
Thus we can write it as: $[x,y,z] = ((xy)z)^{-1} (x(yz))$.
This provides a means to generalise the idea to products of an arbitrary number of octonions,
evaluated in arbitrary orders as we outline in the next section.

\section{Extension to products of higher degree}
\label{sec:extension}
It is obvious that the result in Theorem \ref{thm:associator} can be extended to products of
more than three octonions which take the form $(\cdots(x_1 x_2) x_3) x_4 \ldots)\cdots)$ or similar
with the parentheses arranged from the opposite side.

However, there are other possible evaluation orders for products of 4 terms or more.
Consider the case of four terms, for example.
There are five possible orders of evaluation of four octonions in the order $w, x, y, z$:
\begin{align*}
p_1 &= (((wx)y)z)\\
p_2 &= (wx)(yz)\\
p_3 &= (w(x(yz)))\\
p_4 &= (w(xy))z\\
p_5 &= w((xy)z)
\end{align*}
Clearly we could define an associator in similar manner to Theorem \ref{thm:associator}
between any two of these orders.
The value of this associator can be calculated as follows:
\[
a_{ij} = p_i^{-1} p_j, \quad i, j \in\{1,2,3,4,5\}
\]
By the same reasoning as in the last paragraph of section \ref{sec:intro},
$a_{ji} = \conjugate{a_{ij}}$.

Although this provides a means to calculate a particular associator, it does not yield any insight
into the properties of the five different associators and the relations, if any, between them.

\section{Conclusion}
The multiplicative associator presented in Theorem \ref{thm:associator} is a possible additional
tool for working with octonion expressions and algorithms for computation with octonions.
There is clearly more work possible to expand the results in section \ref{sec:extension}.


\end{document}